\documentclass[12pt]{amsart}
\usepackage{fullpage}
\usepackage[utf8]{inputenc}
\usepackage{amsmath, amssymb, amsthm}
\usepackage{bookmark}
\usepackage{amsthm}

\newtheorem{thm}{Theorem}[section]
\newtheorem{lem}[thm]{Lemma}

\newtheorem{defn}[thm]{Definition}

\newtheorem*{prob}{Problem}
\theoremstyle{definition}

\renewcommand{\Re}{\mathbb{R}}
\newcommand{\Red}{\Re^d}

\newcommand{\norm}[1]{\left\|#1\right\|}
\newcommand{\sch}[1]{\norm{#1}_{C_p^d}}
\newcommand{\Sch}{\mathcal{P}^d}
\newcommand{\EE}{{\mathbb E}}
\newcommand{\EEop}{\mathop{\mathbb E}}

\newcommand{\tr}{\mathrm{trace}}
\newcommand{\conv}{\mathrm{conv}}
\newcommand{\bd}[1]{\mathrm{bd}\left(#1\right)}
\newcommand{\iprod}[2]{\left\langle#1,#2\right\rangle}
\newcommand{\Pcal}{{\mathcal P}}

\DeclareMathOperator{\inter}{int}
\newcommand{\pol}{^\circ}
\newcommand{\iik}{i\in[k]}
\newcommand{\iim}{i\in[m]}
\newcommand{\Qs}{Q_1,\ldots,\allowbreak Q_k}

\newcommand{\qs}{q_1,\ldots,\allowbreak q_k}
\newcommand{\qbars}{\bar q_1,\ldots,\allowbreak \bar q_k}
\newcommand{\rj}{\mathrm{r}}
\newcommand{\st}{:\;}
\newcommand{\noshow}[1]{}

\newcommand\invisiblesection[1]{%
  \refstepcounter{section}%
  \addcontentsline{toc}{section}{\protect\numberline{\thesection}#1}%
  \sectionmark{#1}}

\title{Approximation of the average of some random matrices}
\author{Grigory Ivanov\address{Grigory Ivanov: Institute of Discrete 
Mathematics and 
Geometry, TU Wien, Vienna; Moscow Inst. of Physics and Technology, 
Moscow, Russia}
\and
M\'arton Nasz\'odi\address{M\'arton Nasz\'odi: Alfr\'ed R\'enyi Inst. of Math.; 
MTA-ELTE Lend\"ulet Combinatorial Geometry Research Group;
Dept. of Geometry, Lor\'and E\"otv\"os University, Budapest}
\and
Alexandr Polyanskii\address{Alexandr Polyanskii: Moscow Inst. of Physics and 
Technology, 
Moscow, Russia; Institute for Information Transmission Problems RAS, Moscow, 
Russia; Caucasus Mathematical Center, Adyghe State University, Maikop, Russia}
}

\begin{document}

\begin{abstract}
Rudelson's theorem states that if for a set of unit vectors $u_i$ and positive 
weights $c_i$, we have that $\sum c_i u_i\otimes u_i$ is the identity operator 
$I$ on ${\mathbb R}^d$, then the sum of a \emph{random sample} of $Cd\ln d$ of 
these diadic products is close to $I$. The $\ln d$ term cannot be removed.

On the other hand, the recent fundamental result of Batson, Spielman and 
Srivastava and its improvement by Marcus, Spielman and Srivastava show that the 
$\ln d$ term can be removed, if one wants to show the \emph{existence} of a 
good approximation of $I$ as the average of a few diadic products. It is known 
that 
essentially the same proof as Rudelson's yields a more general statement about 
the average of \emph{positive semi-definite matrices}. 

First, we give an example of an average of positive semi-definite matrices 
where \emph{there is no approximation} of this average by $Cd$ elements. Thus,
the result of Batson, Spielman and Srivastava cannot be extended to this wider 
class of matrices.

Next, we present a stability version of Rudelson's result on positive 
semi-definite matrices, and thus, extend it to certain non-symmetric matrices. 
This yields applications to the study of the Banach--Mazur distance of convex 
bodies. 

Finally, we show that in some cases, one needs to take a subset of the 
vectors of order $d^2$ to approximate the identity.
\end{abstract}

\maketitle

\section{Introduction}
For vectors $u,v\in\Red$, their \emph{tensor product} (or, diadic product) is a 
linear operator on $\Red$ defined as $(u\otimes v)x=\iprod{u}{x}v$ for every 
$x\in\Red$, where $\iprod{u}{x}$ denotes the standard inner product. For a 
positive integer $k$, we use the notation $[k]=\{1,\ldots,k\}$, and the 
\emph{cardinality} of a multi-set $\sigma$ (counting multiplicities) is denoted 
by $|\sigma|$.

A random vector $v$ in $\Red$ is called \emph{isotropic}, if 
$\EE v\otimes v=I$, 
where $\EE$ denotes the expectation of a random variable, and $I$ is the 
identity operator on $\Red$.

According to Rudelson's theorem \cite{Ru99}, if we take $k$ independent copies 
$y_1, \dots , y_k$ of an isotropic random vector $y$ in $\Red$ for which 
$|y|^2\leq\gamma$ almost surely, with
\begin{equation*}%\label{eq:kroughly}
 k=\left\lceil\frac{c\gamma\ln d}{\varepsilon^2}\right\rceil,
\mbox{ then }\; 
\EE\norm{\frac{1}{k}\sum_{i=1}^k y_i\otimes y_i - I}\leq\varepsilon,
\end{equation*}
where $\norm{A}=\max\{\iprod{Ax}{Ax}^{1/2}\st x\in\Red, \iprod{x}{x}=1\}$ 
denotes the \emph{operator norm} of the matrix $A$.
We say that a sequence of unit vectors $u_1,\ldots,u_m$ in $\Red$ yields a 
\emph{John decomposition} of $I$, if $\frac{1}{d}I\in\conv\{u_i\otimes u_i\st 
\iim\}$, that is, if there are scalars $\alpha_1,\ldots,\alpha_m\geq0$ with 
$\sum_{i=1}^m \alpha_i=1$ such that 
\begin{equation}\label{eq:john}
 \sum_{i=1}^m \alpha_i u_i\otimes u_i=\frac{1}{d}I.
\end{equation}

Rudelson's result applies in this setting as well. 
The coefficients $\alpha_i$ define a probability distribution on $[m]$. Let 
$\sigma=\{i_1,\ldots,i_k\}$ be 
a multiset obtained by $k$ independent draws from $[m]$ according to this 
distribution, and consider the following average of matrices 
$\frac{1}{k}\sum_{i\in\sigma} \sqrt{d}u_{i}\otimes \sqrt{d}u_{i}$. It follows 
that, in expectation, this average is not farther than $\varepsilon$ from $I$ 
in the operator norm, provided that $k$ is at least $\frac{cd\ln 
d}{\varepsilon^2}$, where $c$ is some constant.

Our starting point is an observation according to which Rudelson's proof yields 
the following more general statement.
\begin{thm}\label{prop:rudelson}
Let $0<\varepsilon<1$ and 
$Q_1,\ldots,Q_k$ be independent random matrices 
distributed according to (not necessarily identical) probability 
distributions $\Pcal_1,\ldots,\Pcal_k$ on the set $\Sch$ of $d\times d$ real 
positive semi-definite matrices such that $\EE Q_i=A$ for some $A\in\Sch$  and 
all $\iik$. Set $\gamma=\EE(\max_{\iik}\norm{Q_i})$, and assume 
that
\[k\geq \frac{c\gamma(1+\norm{A})\ln d}{\varepsilon^2},\]
where $c$ is an absolute constant.
Then
\begin{equation}
\label{eq:rudelson_ext_theorem}
\EE\norm{\frac{1}{k}\sum_{\iik} Q_i - A}\leq\varepsilon.
\end{equation}
\end{thm}
This result is not new.
Oliveira mentions that \cite{Ru99} implicitly proves Theorem~1 of \cite{Ol10}, 
which in turn easily yields Theorem~\ref{prop:rudelson}. See also 
\cite[Theorems~5.1 and 4.1]{Tr16},  and similar results in \cite{CGT12}.
For completeness, we provide a proof that follows closely Rudelson's argument
\cite{Ru99}.
%: first, we apply symmetrization by Rademacher variables, and then 
%use the Lust-Picard inequality. 

As was noticed by Rudelson \cite[Remark~3.3]{Ru97}, and by Aubrun in 
\cite{aubrun2007sampling}, the logarithmic term in Theorem~\ref{prop:rudelson} 
cannot be omitted. More specifically, even if the $Q_i$ are diadic products, 
\eqref{eq:rudelson_ext_theorem} does not hold for 
$k=\frac{c\gamma(1+\norm{A})}{\varepsilon^2}$.

If instead of considering the \emph{expectation} of the average of randomly 
chosen $u_i\otimes u_i$, we want to show the \emph{existence} of a small subset 
of the set of $u_i\otimes u_i$ whose average is close to $I$, then the picture 
changes, as was shown by a 
completely different approach introduced in the fundamental paper of Batson, 
Spielman and Srivastava \cite{BSS14}. It was developed further by Marcus, 
Spielman and Srivastava \cite{MSS15} (see also \cite{S12}), and by Friedland 
and Youssef \cite{FY17}. In \cite{FY17}, it is shown that if a sequence of unit 
vectors $u_1,\ldots,u_m$ in $\Red$ yields a John decomposition of $I$, then 
there is a (deterministically obtained) multi-subset $\sigma$ of $[m]$ of size 
$|\sigma|=\frac{cd}{\varepsilon^2}$ with 
$\norm{\frac{1}{|\sigma|}\sum_{i\in\sigma} 
\sqrt{d}u_i\otimes \sqrt{d}u_i-I} <\varepsilon$.

The first contribution of the present paper is that we cannot remove the $\ln 
d$ 
term in the setting of Theorem~\ref{prop:rudelson}, even when we choose the 
matrices $Q_i$ deterministically. In fact, in general, there \emph{does not 
exist} a good approximation of size $O(d)$.

\begin{thm}\label{theorem:rudelsonneedslog}
For any integer $d\geq 8$, any $1\leq \gamma$, and any 
$0<\varepsilon<\frac{1}{16}$, 
there are positive semi-definite matrices $Q_1,\ldots,Q_n$ in $\Sch$ with 
$I\in\conv\{Q_i\}$ and $\norm{Q_i}\leq 2 \gamma$ 
%and $\tr(Q_i)=d$ 
for all $i\in [n]$ such that, 
for any non-empty multi-subset $\sigma$ of $[n]$ of size 
$|\sigma|\leq \gamma \frac{\lfloor\log_2 
d\rfloor}{96\varepsilon}$,
we have 
\[
	\norm{\frac{1}{|\sigma|}\sum_{i\in \sigma} Q_i - I} \geq \varepsilon.
\]
\end{thm}

In view of Theorem~\ref{theorem:rudelsonneedslog}, it would be highly 
desirable to unify Rudelson's bound (Theorem~\ref{prop:rudelson}) that applies 
to matrices of any rank and the Batson-Spielman-Srivastava bound which applies 
to matrices of rank one. We propose the following problem.
\begin{prob}
Estimate the function $f(r)$ for which the following holds. 

For any dimension $d$ and error $0<\varepsilon<1$, if 
$Q_1,\ldots,Q_m\in\Re^{d\times d}$ are positive semi-definite matrices of rank 
at most $r$ and operator norm at most $\gamma$ with 
$\frac{1}{m}\sum\limits_{i=1}^m Q_i=I$,
then there is a
multi-subset $\sigma$ of $[m]$ of size 
$|\sigma| \leq\frac{(1+\gamma)f(r)}{\varepsilon^2}$ with 
$\norm{\frac{1}{|\sigma|}\sum_{i\in\sigma} Q_i-I} <\varepsilon$.
\end{prob}
Rudelson's result shows that $f(d)\leq c\ln d$, which by 
Theorem~\ref{theorem:rudelsonneedslog} is sharp.
By \cite{BSS14} and subsequent works, we have $f(1)=c$, which is clearly sharp.
In \cite{FY17} the \emph{stable rank} (ie. the square of the ratio of the 
Hilbert--Schmidt norm and the operator norm) is considered in place of rank. 
Therein, as well as in \cite{silva2015sparse}, it is shown that the $\ln d$ 
term may be removed in certain cases.
Nevertheless, to our knowledge, even for $f(2)$, no bound is known better than 
the obvious $f(2)\leq f(d)\;\;(\leq c\ln d)$.

Our second goal is to study extensions of Theorem~\ref{prop:rudelson} to the 
case of non-symmetric matrices. First, the geometric motivation 
behind the study of these questions in linear algebra comes from John's theorem 
\cite{J48}, extended by K.~Ball \cite{B92} (see also \cite{B97}).
\begin{thm}[John's theorem]
\label{thm:ball_theorem} 
For every convex body $K\subset \Red$, there is a unique ellipsoid of maximum 
volume contained in $K$. Moreover, this ellipsoid is the $o$-centered Euclidean 
unit ball $B_2^d$ if and only if there are contact points $u_1,\ldots,u_m\in\bd 
K\cap\bd{B_2^d}$ such that for some scalars $\alpha_1,\ldots,\alpha_m>0$, 
equation \eqref{eq:john} and $\sum_{i=1}^m \alpha_i u_i=0$ hold.
\end{thm}
Giannopoulos, Perissinaki, and Tsolomitis proved \cite{GPT01} (see also 
\cite{BR02,L79}, \cite[Theorem~14.5]{TJ89}, and for an improved version 
\cite{GLMP04}) that the maximum volume affine image of any convex body $K$ 
contained in $L$ also yields a decomposition of the identity similar to John's. 
In order to state it, we recall some terminology.

The \emph{polar} of a convex body $K$ in $\Red$ is defined as 
$K\pol=\{\ x\st \iprod{x}{y}\leq 1 \mbox{ for every } y\in K\}$.
\begin{defn}\label{def:john}
	Let $K$ and $L$ be convex bodies in $\Red$. We say that $K$ is in 
\emph{John's position} in $L$ if $K\subseteq L$ and for some scalars 
$\alpha_1,\ldots,\alpha_m>0$ with $\sum_{i=1}^m \alpha_i =1$, we have 
	\begin{equation}\label{eq:H1}
	\frac{1}{d}I=\sum_{i=1}^m \alpha_i u_i\otimes v_i
	\end{equation}
	and
	\begin{equation}\label{eq:H2}
	0=\sum_{i=1}^m \alpha_i u_i=\sum_{i=1}^m \alpha_i v_i,
	\end{equation}
	where $u_1,\ldots,u_m\in \bd{L}\cap \bd{K}, 
v_1,\ldots,v_m\in\bd{L\pol} \cap \bd{K\pol}$ with 
$\iprod{u_i}{v_i}=1$ for all $\iim$.
\end{defn}

Note that if $K$ and $L$ are origin-symmetric and \eqref{eq:H1} is satisfied 
for a set of vectors then by including the opposite of each vector, 
\eqref{eq:H2} is also satisfied.

Finally, we can recall Theorem~3.8 from \cite{GLMP04}, which is our geometric 
motivation for extending Rudelson's result to non-symmetric matrices.
\begin{thm}[Gordon, Litvak, Meyer, Pajor] 
Let $K$ and $L$ be two 
convex bodies in $\Red$ such that $K\subseteq L$, and $K$ is of maximum volume 
among all affine images of $K$ contained in $L$. Assume also that $0\in\inter 
L$. 

Then there exists $z\in\inter(K)$ such that $K-z$ is in John's position in 
$L-z$ with $m\leq n^2+n$.
\end{thm}

\begin{defn}\label{defn:johnradius}
	Let  $K$ be a convex body in $\Red$. We denote the Banach--Mazur 
distance of 
$K$ to the Euclidean ball by 
	\begin{equation*}
		\rj(K) = \inf\{\lambda \st \mathcal{E} \subset K-a \subset 
\lambda 
\mathcal{E}, \textnormal{ for some ellispoid } \mathcal{E} \textnormal{ and 
vector } a 
\textnormal{ in } \Red\}.
	\end{equation*}
\end{defn}

It follows from John’s theorem that $\rj(K) \leq d $ for any convex body $K$ in 
$\Red$, and $ 
\rj(K) \leq \sqrt{d}$ for all centrally-symmetric convex bodies.

%Moreover, for the unit balls of $\ell_p$ spaces, we have 
%\begin{equation*}%\label{eqn:RK_for_Lp}
%\rj(K)\leq \begin{cases}
%            d^{1/p - 1/2}& \mbox{ for } p \in [1,2],\\
%            d^{1/2 - 1/p}& \mbox{ for } p \in [2, \infty).
%           \end{cases}
%\end{equation*} 

We prove a stability version of Rudelson's result, that is, when $K$ is very 
close to the Euclidean ball, then we can approximate the identity with diads 
coming from $O(d\ln d)$ contact pairs.

\begin{thm}\label{thm:RudelsonStabilityVeryCloseToBall}
Let $K$ and $L$ be convex bodies in $\Red$ with $B_2^d\subseteq K\subseteq 
(1+1/d^2)B_2^d$.
Assume that $K$ is in John's position in $L$, and the vectors $u_i$ and $v_i$ 
for $i 
\in [m]$ satisfy the conditions of Definition \ref{def:john}. 

Then for any  $\varepsilon \in (0,1)$ and 
\[	
k\geq \frac{cd\ln d}{\varepsilon^2},
\]
where $c>0$ is a universal constant, there is a multiset $\sigma \subset [m]$ 
of size $k$ such that 
\begin{equation}\label{eq:nonsymmapproxjohn}
\norm{\frac{d}{k} \sum\limits_{i \in \sigma}  u_i \otimes v_i - I} \leq 
\varepsilon 
\end{equation}
and
\begin{equation}\label{eq:approxbalanced}
	\frac{1}{k}\norm{ \sum\limits_{i \in \sigma} u_i} \leq 
\frac{\varepsilon}{\sqrt{d}}
\quad\mbox{ and }\quad
	\frac{1}{k}\norm{ \sum\limits_{i \in \sigma} v_i} \leq 
\frac{\varepsilon}{\sqrt{d}}. 
\end{equation}
\end{thm}

On the other hand, when $K$ is not so close to the Euclidean ball, the 
existence of an approximation of $I$ using only a few vector-pairs cannot be 
guaranteed.

\begin{thm}\label{thm:BM_lowerbound}
For any positive integer $d > 2$, any $0<\varepsilon< 1/2$, and  
$0<\delta <\sqrt{d/2 -1}$, there is an origin-symmetric convex body $K\subseteq
C=[-1,1]^d$ with $r(K)=\sqrt{1+\delta^2}$ such that its maximum volume 
inscribed ellipsoid is 
$B_2^d$ and there are points $u_i\in \bd K\cap \bd C$, $v_i\in \bd {K\pol}\cap 
\bd {C\pol}$ satisfying \eqref{eq:H1} with the following  property. If 
$M\subset [m]$ is a subset such that
\[
 \left\|\sum_{i\in M} \beta_i u_i \otimes v_i-I\right\|\leq \varepsilon,
\]
for some  scalars $\beta_i$, then 
\[
|M| \geq d  \min\left\{ \frac{d}{4}, \left(\frac{\delta}{4 \varepsilon}  
\right)^2 
\right\}.
%\geq   \left( \frac{\rj(K)}{4\varepsilon}  \right)^2 d.
\]
\end{thm}

We prove Theorem~\ref{prop:rudelson} in Section~\ref{sec:Rudelson}.
In Section~\ref{sec:nonSymmRudelson}, we show 
Theorem~\ref{thm:RudelsonStabilityVeryCloseToBall}, where the main idea is to 
symmetrize matrices. Its use is limited though, as shown by an example we 
describe in Section~\ref{subsec:symmwonthelp}. This explains why we require $K$ 
to be so close to the ball.
Section~\ref{sec:elloneembedding} contains the proof of 
Theorem~\ref{theorem:rudelsonneedslog}.
Finally, in Section~\ref{sec:slowfar}, we prove Theorem~\ref{thm:BM_lowerbound}.

\section{Symmetric matrices -- Proof of Theorem~\ref{prop:rudelson}}
\label{sec:Rudelson}

Let $\Sch$ denote the cone of positive semi-definite symmetric 
matrices in $\Re^{d\times d}$. The \emph{Schatten $p$-norm} of a real $d\times 
d$ matrix $A$ is defined as 
\[
\sch{A}:=\left(\sum_{i=1}^d 
(s_i(A))^p\right)^{1/p},
\]
where $s_1(A),\ldots,s_d(A)$ is the sequence 
of eigenvalues of the positive semi-definite matrix $\sqrt{A^{\ast}A}$.
We recall that $\norm{A}\leq\sch{A}$ for all $p\geq 1$, and we also have
\begin{equation}\label{eq:schattenvsoperatornorm}
\norm{A}\leq\sch{A}\leq e \norm{A} \mbox{ for } p=\ln d,
\end{equation}
where $\ln$ denotes the natural logarithm and $e$ denotes its base.

From this point on, $\mathbf{r}$ denotes a sequence of $k$ \emph{Rademacher 
variables}, that is, $\mathbf{r}=(r_1,\ldots,r_k)$, where the $r_i$ are random 
variables 
uniformly distributed on $\{1,-1\}$, independent of each other and all other 
random variables in the context.

We state the following inequality due to Lust--Piquard and Pisier 
\cite{Lu86,Lup91}, essentially in the form as it appears in the book 
\cite[Theorem~8.4.1]{Pi98}.
\begin{thm}[Lust--Piquard]\label{thm:LP}
$2\leq p<\infty$. For any $d$ and any $Q_1,\ldots,Q_k$ (not necessarily 
positive definite) square matrices of size $d$ we have
\[
\left[\EEop_{\mathbf{r}}
\sch{\sum_{j=1}^k r_jQ_j}^p\right]^{1/p}\leq 
c\sqrt{p}\max
\left\{
\sch{\left(\sum_{j=1}^k Q_jQ_j^\ast\right)^{1/2}},
\sch{\left(\sum_{j=1}^k Q_j^\ast Q_j\right)^{1/2}}
\right\}
\]
for a universal constant $c>0$.
\end{thm}

Note that for any $d\times d$ matrix $Q$, the product $Q^\ast Q$ is
positive semi-definite. Since, by Weyl's inequality, the Schatten $p$-norm is 
monotone on the cone of positive semi-definite matrices, we may deduce from the 
theorem of Lust--Piquard the following inequality
\begin{equation}\label{eq:LP}
\left[\EEop_{\mathbf{r}}
\sch{\sum_{j=1}^k r_jQ_j}^p\right]^{1/p}\leq 
c\sqrt{p}
\sch{\left(\sum_{j=1}^k Q_jQ_j^\ast+Q_j^\ast Q_j\right)^{1/2}}.
\end{equation}

\begin{lem}[Symmetrization by Rademacher variables]
 \label{lem:symmetrizationrademacher}
Let $q_1,\ldots,q_k$ be independent random vectors 
distributed according to (not necessarily identical) probability 
distributions $\Pcal_1,\ldots,\allowbreak\Pcal_k$ on a normed space $X$ with
$\EE q_i=q$ for all $\iik$.

Then
\[
\EEop_{q_1,\ldots,q_k}\norm{
 \frac{1}{k}\sum_{\ell=1}^k q_{{\ell}} -q}\leq
\frac{2}{k}\EEop_{q_1,\ldots,q_k}\EEop_{\mathbf{r}}
\norm{\sum_{\ell=1}^k r_{\ell} q_{{\ell}}}.
\]
\end{lem}

\begin{proof}[Proof of Lemma~\ref{lem:symmetrizationrademacher}]
Let $\qbars$ be independent random vectors chosen according to 
$\Pcal_1,\ldots,\allowbreak\Pcal_k$, respectively.
\[
\EEop_{\qs}\norm{ \frac{1}{k}\sum_{\ell=1}^k q_{{\ell}} -q}=
\frac{1}{k}\EEop_{\qs}\norm{\sum_{\ell=1}^k 
q_{{\ell}}-\EEop_{\qbars}\sum_{\ell=1}^k \bar q_{\ell}}
\]
\[
\leq
\frac{1}{k}\EEop_{\mathbf{r}}\EEop_{\qs}\EEop_{\qbars}\norm{\sum_{\ell=1}^k 
r_{\ell}(q_{{\ell}}-\bar q_{{\ell}})}
\leq\frac{2}{k}\EEop_{\qs}\EEop_{\mathbf{r}}\norm{\sum_{\ell=1}^k r_{\ell} 
q_{{\ell}}},
\]
where we used that $r_{\ell}(q_{{\ell}}-\bar q_{{\ell}})$ and 
$q_{{\ell}}-\bar q_{{\ell}}$ have the same distribution.
\end{proof}

\begin{proof}[Proof of Theorem~\ref{prop:rudelson}]
The argument follows very closely Rudelson's.

Denote by $D=\frac{1}{k}\sum_{\iik} Q_i - A$, and $p=\ln d$.
Then
\[
\EEop_{\Qs}\norm{D}\leq \EEop_{\Qs}\sch{D}
\stackrel{\mbox{(S)}}{\leq}
\frac{2}{k}\EEop_{\Qs}\EEop_{\mathbf{r}}\sch{\sum_{\ell=1}^k r_{\ell} 
Q_{{\ell}}}
\]
\[
\stackrel{\mbox{(H)}}{\leq}\frac{2}{k}\EEop_{\Qs}\left[\EEop_{\mathbf{r}}\sch{
\sum_{\ell=1}^k r_{\ell} 
Q_{{\ell}}}^p\right]^{1/p}
\stackrel{\mbox{(L-P)}}{\leq}
\frac{c_0\sqrt{p}}{k}\EEop_{\Qs}\sch{\left(\sum_{\ell=1}^k 
Q_{\ell}^2\right)^{1/2}}
\]
\[
\stackrel{\mbox{(PSD)}}{\leq}
\frac{c_0\sqrt{p}}{k}\EEop_{\Qs}\left[
\max_{\ell\in[k]}\norm{Q_{\ell}}^{1/2}\cdot
\sch{\left(\sum_{\ell=1}^k 
Q_{\ell}
\right)^{1/2}}\right]
\]
\[
\leq
\frac{c_1\sqrt{p}}{k}\EEop_{\Qs}\left[
\max_{\ell\in[k]}\norm{Q_{\ell}}^{1/2}\cdot
\norm{\left(\sum_{\ell=1}^k 
Q_{{\ell}}\right)}^{1/2}\right]
\]
\[
\stackrel{\mbox{(H)}}{\leq}
\frac{c_1\sqrt{\gamma p}}{k}\left[
\EEop_{\Qs}\norm{\left(\sum_{\ell=1}^k 
Q_{{\ell}}\right)}\right]^{1/2}
 \leq
\\
\frac{c_1\sqrt{\gamma p}}{\sqrt{k}} 
\left[\EEop_{\Qs}\norm{D}+\norm{A}\right]^{1/2},
\]
where  $c_0$ and $c_1$ are positive constants. Here, we use 
Lemma~\ref{lem:symmetrizationrademacher} in step (S) and the inequality 
\eqref{eq:LP} in step (L-P). The inequality (PSD) relies on the fact that the 
matrices $Q_i$ are positive semi-definite, and (H) follows from H\"older's 
inequality.

Thus, we obtain
\begin{equation*}%\label{eq:wegothere}
\EE\norm{D}\leq \frac{c_1\sqrt{\gamma\ln 
d}}{\sqrt{k}}\sqrt{\EE\norm{D}+\norm{A}}.
\end{equation*}
Denoting by $\alpha=\left(\frac{c_1\sqrt{\gamma\ln d}}{\sqrt{k}}\right)^2$, 
we have
\[
(\EE\norm{D})^2-\alpha\EE\norm{D}-\alpha \norm{A}\leq 0.
\] 
Therefore, we get 
$\EE\norm{D}\leq \alpha+\sqrt{\alpha \norm{A}}$, and thus the inequality
\[
\EE\norm{D}\leq \frac{c_1^2\gamma\ln d}{k}+\frac{c_1\sqrt{\gamma \norm{A} \ln 
d}}{\sqrt{k}}\leq \varepsilon
\]
holds for $k\geq \frac{c\gamma(1+\norm{A})\ln d}{\varepsilon^2}$ with 
sufficiently large $c$. 
Theorem~\ref{prop:rudelson} is proved.
\end{proof}

\section{Upper bound  for non-symmetric diads -- Proof of 
Theorem~\ref{thm:RudelsonStabilityVeryCloseToBall}}\label{sec:nonSymmRudelson}

Theorem~\ref{thm:RudelsonStabilityVeryCloseToBall} is an immediate corollary to 
the following result.

\begin{thm}\label{thm:BM_dist_norm_estimation} 
Let $K$ and $L$ be convex bodies in $\Red$ and $\rj(K)\leq 2$ such that
\[
B_2^d \subseteq K\subseteq r(K) B_2^d.
\]
Assume that $K$ is in John's position in $L$, and the vectors $u_i$ and $v_i$ 
for $i \in [m]$  satisfy the conditions of Definition \ref{def:john}. 

Then, for any  $0<\varepsilon<1$ and 
\[	
k\geq \frac{cd\ln d}{\varepsilon^2}\rj(K)\left(d\sqrt{\rj(K)-1}+1\right),
\]
where $c>0$ is a universal constant,
there is a multiset $\sigma \subset [m]$ of size $k$ such that
\eqref{eq:nonsymmapproxjohn} and \eqref{eq:approxbalanced} 
hold.
\end{thm}

We will show that Theorem~\ref{thm:BM_dist_norm_estimation} follows from the 
following more general result. 

\begin{thm}\label{thm:non-symmetric}
Let $0<\varepsilon<1$ be given, and let $Q_1, \ldots, Q_m$ and $A$ 
be square matrices of size $d$ such that 
\[
    	A = \sum_{i=1}^m\alpha_i Q_i,
\]
where $\alpha_i\geq 0$ and $\sum_{i=1}^m\alpha_i=1$. 
Set $\gamma = \max_i \norm{Q_i}$, and
\[
b = \frac{1}{2\gamma} \norm{\sum_{i=1}^m \alpha_i \left(Q_i Q_i^{\ast} + 
Q^{\ast}_i Q_i \right)}.
\]
Assume that
\[
k\geq \frac{c\gamma(1 + b)\ln d}{\varepsilon^2}
\] 
for some constant $c>0$.

Then there is a multi-subset $\sigma \subset[m]$ of size $k$ such that
\[
\norm{\frac{1}{k}\sum_{i\in \sigma} Q_i - A}\leq\varepsilon.
\]
\end{thm}

\subsection{Proof of Theorem~\ref{thm:non-symmetric}}
Randomly draw $k$ elements from $[m]$, each time taking each element 
with 
probability $\frac{1}{m}$, and denote by $\sigma$ the random multisubset of 
$[m]$ that is 
obtained.

Set
\[
 U_i = \frac{1}{2\gamma} \left(Q_i Q_i^{\ast} + Q^{\ast}_i Q_i \right),
 \quad
 B =  \sum_{i=1}^m \alpha_i  U_i.
 \quad \mbox{ and  }\quad 
 E_B = \mathop\EE\limits_\sigma \norm{\frac{1}{k } \sum\limits_{j\in 
 \sigma} U_j - B},
\]
Then $b = \norm{B}$, and $\norm{U_i}\leq \gamma$.

Since the $U_i$ are positive semi-definite matrices, we can apply 
Theorem~\ref{prop:rudelson} and get that
$E_B \leq \varepsilon.$

Setting $p=\ln d$, we obtain that
\[
\frac{1}{k}\EEop_{\sigma} \left\{\sch{\left(\sum\limits_{j \in \sigma}   
U_j\right)^{1/2}}\right\} = 
\frac{1}{\sqrt{k}}\EEop_{\sigma} \left\{\sch{\left(\frac{1}{k}\sum\limits_{j 
\in 
\sigma}   U_j\right)^{1/2}}\right\} 
\]
\[
\stackrel{\eqref{eq:schattenvsoperatornorm}}{\leq} 
\frac{e}{\sqrt{k}}\EEop_{\sigma} \left\{\norm{\left(\frac{1}{k}\sum\limits_{j 
\in 
\sigma}   U_j\right)^{1/2}}\right\}  =
\frac{e}{\sqrt{k}}\EEop_{\sigma} \left\{\norm{\frac{1}{k}\sum\limits_{j \in 
\sigma}   U_j}^{1/2}\right\} 
\]
\[
\stackrel{\mbox{(T)}}{\leq} 
\frac{e}{\sqrt{k}}\EEop_{\sigma} \left\{\left(\norm{\frac{1}{k}\sum\limits_{j 
\in 
\sigma}   U_j - B} + \norm{B}\right)^{1/2}\right\}
\]
\begin{equation}\label{eq:symmetricboundhere}
\stackrel{\mbox{(H)}}{\leq} 
\frac{e}{\sqrt{k}} \left(\mathop\EE\limits_\sigma \left( \norm{\frac{1}{k } 
\sum\limits_{j\in \sigma} U_j  - B} + \norm{B} \right)\right)^{1/2} =
\frac{e}{\sqrt{k}} \left(E_B + \norm{B}\right)^{1/2}
 \leq 
 \frac{e \sqrt{1 + b}}{\sqrt{k}}. 
\end{equation}
Here, (T) and (H) follow from the triangle inequality and H{\"o}lder's 
inequality respectively.

Set
\[
D_\sigma=\frac{1}{k}\sum\limits_{j\in \sigma} Q_j - A.
\]
Our aim is to show that $\EEop_{\sigma}\norm{D_\sigma}\leq\varepsilon$.

Using the notation $\mathbf{r}$ for Rademacher variables introduced in 
Section~\ref{sec:Rudelson}, we have
\[
\EEop_{\sigma}\norm{D_\sigma}\leq  \EEop_{\sigma}\sch{D_\sigma} 
\stackrel{\mbox{(S)}}{\leq}
\frac{2}{k}\EEop_{\sigma}\EEop_{\mathbf{r}}\sch{\sum_{i\in \sigma} r_i Q_i}
\stackrel{\mbox{(H)}}{\leq}
\frac{2}{k}\EEop_{\sigma}\left[\EEop_{\mathbf{r}}\sch{\sum_{i\in \sigma} r_i 
Q_i}^p\right]^{1/p}
\]
\[
\stackrel{\mbox{(L-P)}}{\leq}\frac{c_1 \sqrt{p}}{k}\EEop_{\sigma} \left\{
\sch{\left(\sum\limits_{j \in \sigma}  \left[Q_jQ_j^\ast + Q_j^\ast 
Q_j\right]\right)^{1/2}}
\right\}
=\frac{c_1 \sqrt{p} \sqrt{\gamma}}{k}\EEop_{\sigma} 
\left\{\sch{\left(\sum\limits_{j 
\in \sigma}   U_j\right)^{1/2}}\right\} 
\]
\[ 
\stackrel{\eqref{eq:symmetricboundhere}}{\leq}
 \frac{e c_1 \sqrt{p} \sqrt{\gamma} \sqrt{1 + b}}{\sqrt{k}}  \leq  \varepsilon,
\]
where $c_1$ is some positive constant. Note that (S) and (L-P) follow from 
Lemma~\ref{lem:symmetrizationrademacher} and \eqref{eq:LP} respectively and the 
last inequality holds for a sufficiently large constant $c$, and (H) follows 
from H\"older's 
inequality. This finishes the 
proof of Theorem~\ref{thm:non-symmetric}.

\subsection{Proof of Theorem~\ref{thm:BM_dist_norm_estimation}} 
In order to obtain \eqref{eq:nonsymmapproxjohn} and \eqref{eq:approxbalanced} 
for the vectors $u_i$ and $v_i$ in $\Red$, we will prove 
\eqref{eq:nonsymmapproxjohn} for the vectors $a_i=(v_i, 1/\sqrt{d})$ and 
$b_i=(u_i, 1/\sqrt{d})$ in $\Re^{d+1}$. This lifting argument is standard, so 
we leave it to the reader to verify that the latter claim is indeed sufficient.

Since $u_i$ and $v_i$ satisfy Definition~\ref{def:john}, we have
\[
\sum\limits_{i \in m} \alpha_i d a_i \otimes b_i=I_{d+1},
\]
where $I_{d+1}$ is the identity operator on~$\Re^{d+1}$.
Note, that $\iprod{a_i}{b_i}=1+1/d$. We will assume that $d\geq3$.

%\[
%\norm{d a_i \otimes b_i} \leq d \sqrt{\rj^2(K) +1/d} \sqrt{1 + 1/d}\leq 2d 
%\rj(K).
%\] 

For every $i\in[m]$, set $Q_i=d a_i \otimes b_i$. We are going to use the 
notations 
$A, U_i, B,b$, and $\gamma$ as defined in 
Theorem~\ref{thm:non-symmetric}, where $A=I_{d+1}$. 
Since $B_2^d\subseteq K\subseteq \rj(K)B_2^d$, we have 
\[ 
%1+1/d\leq
\norm{a_i}^2\leq\rj(K)^2+1/d\text{\quad
and\quad}
%1/\rj(K)^2+1/d\leq
\norm{b_i}^2\leq 1+1/d.
\] 
Using $\norm{Q_i} = d \norm{a_i} \norm{b_i}$ and $1\leq\rj(K)$, we obtain
\[
%d^2/\rj(K)^2	\leq d^2(1/\rj(K)^2+1/d)(1+1/d)\leq 
\norm{Q_i}^2\leq 
d^2(\rj(K)^2+1/d)(1+1/d)\leq d^2 \rj(K)^2(1+1/d)^2,
\]
and thus
\begin{equation}\label{eq:gammabound}
%d/\rj(K)\leq\gamma \leq d \rj(K)(1+1/d).
\end{equation}
Since $d\norm{a_i}\norm{b_i}\leq \gamma$ and $d/\gamma\geq 1/((1+1/d)\rj(K))$, 
we get
\[
 \norm{\frac{d^2}{\gamma} \norm {a_i}^2b_i\otimes b_i- d a_i\otimes b_i}=
 d\norm{b_i}\norm{\frac{d}{\gamma} \norm{a_i}^2b_i-a_i}
\]\[
=
d \norm{a_i}\norm{b_i}\left[ 
\left(\frac{d\norm{a_i}\norm{b_i}}{\gamma}\right)^2 
+1-\frac{2d(1+1/d)}{\gamma}\right]^{1/2} 
\]
\begin{equation}\label{eq:uuvsuv}
\leq
(1+1/d) d\rj(K) \left(2-2/\rj(K)\right)^\frac12\leq 2d \sqrt{\rj(K)(\rj(K)-1)}. 
\end{equation}
Similarly, we have
\begin{equation}\label{eq:vvvsuv}
 \norm{\frac{d^2}{\gamma} \norm{b_i}^2 a_i\otimes a_i- d b_i\otimes a_i}
 \leq
2d \sqrt{\rj(K)(\rj(K)-1)}. 
\end{equation}

By the definition of $B$, we have
\begin{equation*}
B=\sum_{i=1}^m \alpha_i  U_i
=\sum_{i=1}^m \alpha_i \frac{(Q_i Q_i^{\ast}-\gamma d b_i\otimes 
a_i)+(Q_i^{\ast} Q_i-\gamma d a_i\otimes b_i)}{2\gamma}
+\sum_{i=1}^m \alpha_i \frac{a_i\otimes b_i+b_i\otimes a_i}{2}
\end{equation*}
\begin{equation}
\label{equation:B}
=\sum_{i=1}^m \alpha_i \frac{(Q_i Q_i^{\ast}-\gamma d b_i\otimes 
a_i)+(Q_i^{\ast} Q_i-\gamma d a_i\otimes b_i)}{2\gamma}
+I_{d+1}.
\end{equation}
Using the equations $Q_i 
Q_i^{\ast}=d^2\norm{b_i}^2a_i\otimes a_i$ and $Q_i^{\ast} Q_i=d^2
\norm{a_i}^2b_i\otimes 
b_i$ and inequalities \eqref{eq:uuvsuv} and \eqref{eq:vvvsuv}, we deduce 
from \eqref{equation:B} that
\begin{equation*}
 b=\norm{B}\leq 2d\left(\rj(K)(\rj(K)-1)\right)^{1/2}+1 \leq 4d 
\sqrt{\rj(K)-1}+1,
\end{equation*}
because $\rj(K)\leq 2$. Combining the last inequality with 
\eqref{eq:gammabound}, we apply Theorem~\ref{thm:non-symmetric} with 
$A=I_{d+1}$, and obtain \eqref{eq:nonsymmapproxjohn}, completing the proof of 
Theorem~\ref{thm:BM_dist_norm_estimation}.

\subsection{Symmetrization will not always help}\label{subsec:symmwonthelp}

The main idea of the proof of Theorem~\ref{thm:non-symmetric} is to symmetrize 
the matrices. Its use is limited,  as shown by the following example.

Fix a $\delta>0$, and for $i\in\{1,2\}$, set $u_i=v_i=e_i$. For 
$i\in\{3,\ldots,d\}$, let 
\[
u_i^{+} = e_i + \delta e_1, \quad u_i^{-} = e_i - \delta e_1 ;
\]
and
\[
v_i^{+} =  e_i + \delta e_2, \quad v_i^{-} =  e_i - \delta e_2.
\]
Then clearly,
\[
I = u_1 \otimes v_1 + u_2 \otimes v_2 +  \frac{1}{4}\sum\limits_{3 \leq i 
\leq d} \left( u_i^{+} \otimes v_i^{+} + u_i^{+} \otimes v_i^{-} + u_i^{-} 
\otimes v_i^{+} + u_i^{-} \otimes v_i^{-} \right),
\]
and thus,
\[
I=\frac{1}{4d} \bigg[Q_1+Q_2+Q_3^{++}+Q_3^{+-}+Q_3^{-+}+Q_3^{--}+\ldots+ 
Q_d^{++}+Q_d^{+-}+Q_d^{-+}+Q_d^{--}\bigg],
\]
where $Q_1=4du_1\otimes v_1, Q_2=4du_2\otimes v_2$, and 
$Q_i^{++}=du_i^{+}\otimes v_i^{+}, Q_i^{+-}=du_i^{+}\otimes v_i^{-}$, etc.

Now, the $b$ that appears in Theorem~\ref{thm:non-symmetric} is large. For, say 
$\delta=0.1$, we have $\gamma=16d$, and  $b>0.01d$ (we leave the details to the 
reader), thus, Theorem~\ref{thm:non-symmetric} yields no meaningful result. 
Moreover, the above example can be realized geometrically as contact points (as 
in Definition~\ref{def:john}) of two convex bodies, both constant close to the 
Euclidean ball in the Banach--Mazur distance.

We note however, that it is enough to use only $C d$ vectors to approximate the 
identity, if in Theorem~\ref{thm:non-symmetric} $K$ is an ellipsoid constant 
close to the standard unit ball. 

\section{The log factor is needed -- Proof of 
Theorem~\ref{theorem:rudelsonneedslog}}\label{sec:elloneembedding}

In this section, we prove Theorem~\ref{theorem:rudelsonneedslog}.
First, in Lemma~\ref{lem:noapproxinellone}, we show that in $\ell_1^{t}$, a 
point in the convex hull of other points may not be well approximated in terms 
of the dimension $t$. Then, we use the fact that $\ell_1^{t}$ embeds 
isometrically in $\ell_{\infty}^d$ for $d=2^t$, which embeds isometrically in 
the space of matrices of size $d\times d$.

\begin{lem}\label{lem:noapproxinellone}
Consider the point $a=\frac{1}{12k}(1,\ldots,1)\in \ell^t_1$, where $k$ and $t$ 
are positive integers. Denote by
$e_1,\ldots,e_t$ the standard basis of $\ell_1^t$ and by
$e_{t+1}$ the zero vector.

Then, for any non-empty multiset $\sigma_0\subset [t+1]$ of size $s$, where 
$s\leq 3k$, we have
\[
	\norm{\frac{1}{s}\sum_{i\in \sigma_0} \frac{e_i}{2} - a}_1 \geq 
\frac{t}{12k}.
\] 
\end{lem}
%\begin{remark} 
%In the proof of the theorem, we choose $t=\lfloor \log_2 d\rfloor$.
%\end{remark}
\begin{proof}[Proof of Lemma~\ref{lem:noapproxinellone}]
	Since the $i$-th coordinate $b_i$ of $\frac{1}{s}\sum_{i\in \sigma_0} 
e_i/2$ is either equal to 0 or at least $\frac{1}{2s}\geq \frac{1}{6k}$, we have
$|b_i-\frac{1}{12k}|\geq\frac{1}{12k}$ for every $i\in [t]$, which 
finishes the proof of the lemma.
\end{proof}

Without loss of generality, we may assume that $d=2^t$, where $t$ is a 
non-negative integer. Indeed, assume that we  find the desired matrices $Q_1, 
\dots, Q_n$  for $d=2^t$. Let now $d$ be $d=2^t+h$ with $h<2^t$. Then the 
matrices $Q_1',\dots, Q_n'$ of size $d\times d$ whose upper left corner is 
$Q_i$, and the other entries are zero will satisfy the conditions of the 
theorem. 

Note that it is sufficient to consider multisets $\sigma$ such that
\begin{equation}\label{eq:msigma}
	m/2 <|\sigma|\leq m.
\end{equation}
Indeed, if the theorem is proved for such multisets, then it holds for a 
multiset $\sigma$ with $|\sigma|\leq m/2$: the multiset $\sigma'$ consisting of 
$2^l$ copies of $\sigma$, where $l=\lfloor \log_2 (m/|\sigma|)\rfloor$, 
satisfies~\eqref{eq:msigma}, and thus the statement of the theorem holds for 
$\sigma'$. Since $\sigma'$ consists of several copies of $\sigma$, we can 
easily conclude that Theorem~\ref{theorem:rudelsonneedslog} is true for 
$\sigma$.

Enumerate all $\pm 1$ sequences of length $t$ as 
$s_1,\ldots,s_{d}$. Clearly, the linear map 
\begin{equation*}
 \phi: \ell_1^{t}\rightarrow \ell_{\infty}^d \text{ such that } 
\phi(x)=(\iprod{x}{s_1}, \dots, \iprod{x}{s_d})
\end{equation*}
embeds $\ell_1^{t}$ isometrically into $\ell_{\infty}^d$. We identify 
$\ell_{\infty}^d$ with the subspace of diagonal matrices in the space 
$\Re^{d\times d}$ equipped with the operator norm.

Next, we are going to construct the desired matrices $Q_i$. Let $k$ be an 
integer 
such that 
\begin{equation}
	\label{eq:definition of k}
	k= \left\lfloor \frac{m}{\gamma}\right\rfloor=\left\lfloor 
\frac{t}{96\varepsilon}\right\rfloor,
\end{equation}
so $k\geq 1$. Using the notation introduced in 
Lemma~\ref{lem:noapproxinellone}, for every 
$i\in[t+1]$, put 
\[Q_i=\gamma\psi(e_i/2),\mbox{ where }
\psi(x)=\phi(x-a)+I.
\]
Note that $\psi$ is an affine isometry from $\ell_1^t$ into $\ell_\infty^d$.

By \eqref{eq:definition of k}, we have 
\[
	\left\|\frac{e_i}{2}-a\right\|_1 \leq 
\norm{\frac{e_i}{2}}_1+\norm{a}_1\leq \frac{1}{2}+\frac{t}{12 k}< 
\frac{1}{2}+\frac{96t\varepsilon}{12t}=\frac12 + 8\varepsilon\leq 1,
\]
for every $i\in [t+1]$, and therefore $\norm{Q_i}\leq \gamma\left(\norm{I}+ 
\norm{e_i/2-a}_1\right)<2\gamma$ and the matrix $Q_i$ is positive definite. 

Assuming $\lambda_i=\frac{1}{6k}$ for $i\in[t]$ and 
$\lambda_{t+1}=1-\frac{t}{6k}$, we have
\[
	a=\sum_{i=1}^{t+1}\lambda_i \frac{e_i}{2}.
\] 
Since $\sum_{i=1}^{t+1}\lambda_i=1$ and $\lambda_i\geq 0$ for every 
$i\in[t+1]$, we obtain $a\in \conv \left\{\frac{e_1}{2},\dots, 
\frac{e_{t+1}}{2}\right\}$. Thus, 
denoting by $Q_{t+2}$ the zero matrix, we get
\[
\sum_{i=1}^{t+1}\frac{\lambda_i}{\gamma}Q_i+\left(1-\frac{1}{\gamma}\right)Q_{
t+2}=\sum_{i=1}^{t+1}\lambda_i(\phi(e_i/2-a)+I)= 
\phi\left(\sum_{i=1}^{t+1}\lambda_i e_i/2\right)-\phi(a)+I=I,
\] 
that is,
\[
	I\in\conv\{Q_1,\dots, Q_{t+2}\}.
\]

To prove the theorem, assume that there is a multiset 
$\sigma$ of $[t+2]$ with~\eqref{eq:msigma} such that
\begin{equation}\label{eq:inftynormbd}
		 \norm{\frac{1}{|\sigma|}\sum_{i\in \sigma} Q_i -I}<\varepsilon.
\end{equation}

Next, $(t+2)$ is an element of $\sigma$ with multiplicity $|\sigma|-s$ for 
some non-negative integer $s$. Denote by $\sigma_0$ the multi-subset of 
$\sigma$ which does not contain $(t+2)$. Thus, $|\sigma_0|=s$.

Since $\tr (\phi(y))=0$, we obtain $\tr(Q_i)=\gamma \tr(I)=\gamma d$ for every 
$i\in[t+1]$. Thus, it follows from \eqref{eq:inftynormbd} and the inequality 
$|\tr (A)| \leq d\norm{A}$ for an arbitrary matrix $A$ of size $d$ that
\begin{equation}\label{eq:gammask}
\left|\frac{\gamma s}{|\sigma|}-1\right|<\varepsilon,
%\mbox{ and hence, } 
%\left|\frac{k}{\gamma s}-1\right|<2\varepsilon,
\end{equation}
and therefore
\[
	\frac{|\sigma|}{\gamma} (1-\varepsilon)<s<\frac{|\sigma|}{\gamma} 
(1+\varepsilon).
\]
By \eqref{eq:gammask}, \eqref{eq:msigma} and \eqref{eq:definition of k}, we 
obtain
\begin{equation}\label{eq:skupper}
s<\frac{|\sigma|}{\gamma}(1+\varepsilon)\leq\frac{m}{\gamma} (1+\varepsilon)< 3k
\end{equation}
and 
\begin{equation}\label{eq:sklower}
s>\frac{|\sigma|}{\gamma}(1-\varepsilon)> \frac{m}{2\gamma} (1-\varepsilon)> 
\frac{k}{4}
\end{equation}
Since $\psi$ is an affine map, we have
\[
\norm{\frac{1}{|\sigma|}\sum_{i\in \sigma} Q_i-I}=
\norm{\frac{\gamma}{|\sigma|}\sum_{i\in 
\sigma_0}\psi\left(\frac{e_i}{2}\right)-I}=
\norm{ \frac{\gamma}{|\sigma|}\psi\left(\sum_{i\in\sigma_0} 
\frac{e_i}{2}\right)-I}=
\]
\[
\norm{\frac{\gamma}{|\sigma|}\phi\left(\sum_{i\in \sigma_0} \frac{e_i}{2}  - 
sa\right) +\frac{\gamma s}{|\sigma|}I - I}  \geq
\]
\[
\frac{\gamma s}{|\sigma|}\norm{\phi\left(\frac{1}{s}\sum_{i\in\sigma_0} 
 \frac{e_i}{2}-a\right)}-\left|\frac{\gamma 
s}{|\sigma|}-1\right|\geq
\]
\[
\frac{\gamma s}{|\sigma|}\norm{\frac{1}{s}\sum_{i\in\sigma_0} 
\frac{e_i}{2}-a}_1-\varepsilon,
\]
where, in the last inequality, we combine \eqref{eq:gammask} with the fact that 
$\phi$ is an isometry.

By \eqref{eq:skupper}, we may apply Lemma~\ref{lem:noapproxinellone} to the 
multiset 
$\sigma_0$ to obtain the inequality
\[
\norm{\frac{1}{|\sigma|}\sum_{i\in \sigma} Q_i-I} > 
\frac{\gamma s}{|\sigma|}\cdot\frac{t}{12k}-\varepsilon,
\]
which, by \eqref{eq:inftynormbd} and \eqref{eq:sklower}, yields
\[
	2\varepsilon>\frac{\gamma t}{48|\sigma|},
\]
and thus
\[
	m\geq |\sigma|>\frac{\gamma t}{96\varepsilon},
\]
completing the proof of Theorem~\ref{theorem:rudelsonneedslog}.

\section{Slow approximation far from the ball -- Proof of 
Theorem~\ref{thm:BM_lowerbound}}\label{sec:slowfar}

To prove Theorem~\ref{thm:BM_lowerbound}, let $\{e_1, \ldots, e_d\}$ be the 
standard basis in $\Red$, and set $d^{\prime} = 2^{\lfloor \log_2 d \rfloor}$. 
Clearly, $d^{\prime} \ge d/2 > \delta^2 + 1$.
For every $i \in [d]$ and $j \in \left[d^{\prime} \right]$, let $w_i^j$ be a 
vector satisfying the following conditions:
\begin{enumerate}
\item $w_i^j\in\bd C$, and $\iprod{ w_i^j}{e_i}  = 1$ for all $i \in [d]$ and 
$j \in \left[d^{\prime} \right]$; 
\item $|w_i^j - e_i| = \delta$ for all $j \in [d^{\prime}]$;
\item $w_i^1, \ldots, w_i^{d^{\prime}}$ are the vertices of the 
$(d^{\prime}-1)$-dimensional 
regular simplex centered at $e_i$. In particular, 
$\sum_{j=1}^{d^{\prime}} w_i^j = d^{\prime} e_i$.
\end{enumerate}

In order to prove the existence of $w_i^j$, it is clearly sufficient to show 
the following.
\begin{lem}\label{lem:simplexincube}
 Let $D=2^k$ for some positive integer $k$. Then there is a
$(D-1)$-dimensional regular simplex $\conv\{p^1,\ldots,p^D\}$  contained in the 
$(D-1)$-dimensional cube $\{x\in[-1,1]^D\st \iprod{(1,0,\ldots,0)}{x}=1\}$ with 
$|(1,0,\ldots,0)-p^j|=\sqrt{D-1}$ for all $j\in[D-1]$.
\end{lem}
\begin{proof}[Proof of Lemma~\ref{lem:simplexincube}]
The simplex can be constructed in a standard way using Hadamard matrices as 
follows.
Recall that Walsh matrices (the simplest examples of Hadamard matrices) 
are defined by the following recursive construction
\[
H_{1} ={\begin{bmatrix}1\end{bmatrix}},\quad 
H_{2} ={\begin{bmatrix}1&1\\1&-1\end{bmatrix}},\quad 
\dots, \quad   
H_{2^{k}}={\begin{bmatrix}H_{2^{k-1}}&H_{2^{k-1}}\\H_{2^{k-1}}&-H_{2^{k-1}}
\end{bmatrix}}.
\]
Let $p^1,\ldots, p^D$ be the columns of $H_{2^k}$. They are pairwise 
orthogonal, and each one is of Euclidean length $\sqrt{D}$. They clearly 
satisfy the requirements of Lemma~\ref{lem:simplexincube}.
\end{proof}
%https://mathoverflow.net/questions/38724/coordinates-of-vertices-of-regular-sim
%plex
%Reference therein: 
%https://www.sciencedirect.com/science/article/pii/S0723086903800373?via%3Dihub

Consider the following convex body
\begin{equation}
K=\conv\left(B_2^d\cup \{\pm w_i^j\}_{i \in [d],\; j \in [d^{\prime}]}\right).
\end{equation}
and points $u_{(i-1) d^{\prime} + j}=w_i^j$  and $v_{(i-1) d^{\prime}+ j}=e_i$ 
for 
$i \in [d]$ and $j \in \left[ d^{\prime} \right].$

%Clearly, $\{\pm w_i^j\}_{i,j \in [d]} \subset \bd{K}$ and $\{ e_1, \ldots, e_d 
%\} \subset \bd{K^\circ}$.
We check that the convex body $K$ and the points $u_k, v_k$ (with $k\in 
[dd^{\prime}]$) satisfy 
the conditions of the theorem.
First, by Theorem~\ref{thm:ball_theorem}, the ball $B_2^d$ is the maximum 
volume 
ellipsoid in $K$. Second, $u_k\in \bd K\cap \bd C$ and $v_k\in \bd{K\pol} \cap 
\bd{C\pol}$. 

Since $\sum_{j=1}^{d^{\prime}} w_i^j=d^{\prime}e_i$, we have
\[
I =\frac{1}{d^{\prime}} \sum_{i=1}^d\sum_{j=1}^{d^{\prime}} w_i^j \otimes 
e_i=\frac{1}{d d^{\prime}} 
\sum_{k=1}^{d d^{\prime}} d u_k\otimes v_k.
\]

For simplicity, denote by $Q_{ij}$ the operator $d u_{(i-1) 
d^{\prime}+j}\otimes 
v_{(i-1) d^{\prime} + j}=d w_i^j \otimes e_i$.

Let $M\subset [d]\times [d^{\prime}]$ be a subset of the set of pairs of 
indices and let $\beta_{ij}$ 
be non-zero scalars, where $(i,j)\in M$, such that $A=\sum_{(i,j)\in 
M}\beta_{ij}Q_{ij}$ is an $\varepsilon$-approximation of $I$, that is,
$\|A-I\|\leq \varepsilon$.

Considering $M$ as a $0-1$ matrix of size $d\times d^{\prime}$, we may assume 
that the 
first row contains the smallest number of ones, denote this number by $\ell$.
Since $d^{\prime} \ge d/2,$ it is sufficient to show that 
$\ell\geq\min\left\{\frac{d^{\prime}}{2}, 
\left(\frac{\delta}{4\varepsilon}\right)^2\right\}$.
We will assume that $\ell<\frac{d^{\prime}}{2}$, and will show that 
$\ell\geq\left(\frac{\delta}{4\varepsilon}\right)^2$.

Without loss of generality, assume that $\beta_{1k}=0$ if $(1,k)\not\in M$. 
Since $\norm{A - I} \leq \varepsilon,$ we have
\[
	\varepsilon  \geq \left\|\left(A - I\right) e_1\right\| \geq
	 \left|\iprod{\left(A - I\right) e_1} {e_1}\right|  = 
	 \left|\sum\limits_{j=1}^{d^{\prime}} \beta_{1j} d^{\prime} 
\iprod{w_1^j}{e_1} - 1\right|= 
\left|d^{\prime} \sum\limits_{j=1}^{d^{\prime}} \beta_{1j} - 1\right|. 
\]
Therefore,  
\begin{equation}
\label{eqn:coef_sum_est}
 d^{\prime} \sum\limits_{j=1}^{d^{\prime}} \beta_{1j}  \geq 1-\varepsilon 
\geq \frac{1}{2}.
\end{equation}
Consider the vectors
\[
	y=   {\sum\limits_{j \st \beta_{1j} \neq 0} (w_1^j - 
e_1)} \text{ and } x=y/\|y\|.
\]
Note that $x$ is a unit vector with $\iprod{x}{e_1}= 0$ and
\[ \iprod{(w_1^j - e_1)}{x} = \iprod{w_1^j}{x}=\frac{\delta}{\sqrt{\ell}} 
\sqrt{\frac{d^{\prime} - \ell}{d^{\prime} - 1}} 
\geq 
\frac{\delta}{{\sqrt{2\ell}}}
\]
for all $j$ such that $\beta_{1j}\ne0$. The last inequality holds by 
the assumption $\ell< d^{\prime}/2$.

On the other hand, by \eqref{eqn:coef_sum_est}, 
\begin{gather*}
\varepsilon \geq \norm{\left(A - I\right) x} \geq
\left|\iprod{\left(A - I\right) x}{e_1}\right| = 
\iprod{\left(\sum\limits_{j=1}^{d^{\prime}} \beta_{1j} Q_{1j} \right) x}{e_1}\\
=d\sum\limits_{j=1}^{d^{\prime}} \beta_{1j} \iprod{w_1^j}{x}  = 
d \sum\limits_{j=1}^{d^{\prime}} \beta_{1j}  \iprod{(w_1^j - e_1)}{x} 
\geq 
d^{\prime} \sum\limits_{j=1}^{d^{\prime}} \beta_{1j} \frac{\delta}{{ 
\sqrt{2\ell}}} \geq 
\frac{\delta}{4 
\sqrt\ell}.
\end{gather*}
Thus, $\ell \geq \left(\frac{\delta}{4\varepsilon}\right)^2$ as needed, and the 
proof of Theorem~\ref{thm:BM_lowerbound} is complete.

\invisiblesection{Acknowledgement}
\subsection*{Acknowledgement}
We thank the referees for a number of important citations, as well as for the 
suggestions that improved the presentation.

Part of the research was carried out while the three authors were members of 
J\'anos Pach's chair of DCG at EPFL, Lausanne, which was supported by Swiss 
National Science Foundation Grants 200020-162884 and 200021-175977.

G.I. was supported also by the Swiss National Science Foundation grant 
200021-179133 and by the Russian Foundation for Basic Research, project 
18-01-00036A. M.N. was supported also by the National Research, Development and 
Innovation Fund grant K119670. A.P. was supported in part also by the Leading 
Scientific Schools of Russia through Grant No. NSh-6760.2018.1 (Theorems 1.1 and 
1.2) and by the Ministry of Education and Science of the Russian Federation in 
the framework of MegaGrant no 075-15-2019-1926 (Theorems 1.7 and 1.8).

\bibliographystyle{alpha}
\bibliography{biblio}

\end{document}